\documentclass{article}
\usepackage{arxiv}
\usepackage[utf8]{inputenc} 
\usepackage[T1]{fontenc}    
\usepackage{hyperref}       
\usepackage{url}            
\usepackage{booktabs}       
\usepackage{amsfonts}       
\usepackage{nicefrac}       
\usepackage{lipsum}
\usepackage[english]{babel}

\usepackage{amsfonts}
\usepackage{graphics}
\usepackage{dsfont}
\usepackage{bm,bbm}
\usepackage{amsmath,amssymb,amsthm}
\usepackage{mathtools}
\usepackage{algorithm, algorithmic}
\usepackage{pifont}
\usepackage{cases}
\usepackage{subcaption,graphicx}
\usepackage{stackengine}    
\usepackage{algorithm}
\usepackage{algorithmic}

\newtheorem{theorem}{Theorem}[section]

\renewcommand{\le}{\leqslant}
\renewcommand{\ge}{\geqslant}

\DeclareMathOperator*{\argmin}{arg\,min}
\DeclareMathOperator*{\argmax}{arg\,max}

\newcommand*{\tbar}{\bar{t}}
\newcommand*{\fbar}{\bar{f}}

\newcommand*{\R}{\mathbb{R}}

\begin{document}

\title{An evolutionary view on equilibrium models of transport flows
}

\author{
    Evgenia Gasnikova \\
    Moscow Institute of Physics and Technology \\
    Dolgoprudny, Russia \\
    \texttt{egasnikov@yandex.ru} \\
    \And
    Alexander Gasnikov \\
    Moscow Institute of Physics and Technology \\
    Dolgoprudny, Russia \\
    \texttt{gasnikov@yandex.ru}
    \And
    Yaroslav Kholodov \\
    Innopolis University \\
    Innopolis, Russia \\
    \texttt{ya.kholodov@innopolis.ru}
    \And
    Anastasiya Zukhba \\
    Moscow Institute of Physics and Technology \\
    Dolgoprudny, Russia \\
    \texttt{azukhba@yandex.ru}
}

 \renewcommand{\today}

\maketitle


\begin{abstract}
	In this short paper we describe natural logit population games dynamics that explain equilibrium models of origin-destination matrix estimation and (stochastic) traffic assignment models (Beckmann, Nesterov--de Palma). Composition of the proposed dynamics allows to explain two-stages traffic assignment models.
\end{abstract}

\section{Introduction}
The first traffic assignment model was proposed for about 70 years ago in the work of M.~Beckmann \cite{beckmann1952continuous}, see also \cite{beckmann1956studies}. Nowadays Beckmann's type models are rather well studied \cite{patriksson2015traffic,sheffi1985urban,nesterov2003stationary,baimurzina2019universal,gasnikov2020traffic}. The entropy based origin-destination matrix models are also well developed nowadays \cite{wilson2013entropy,gasnikov2013vvedenie,gasnikov2020traffic}. Moreover, as it was mentioned in \cite{gasnikov2016evolutionary} both of these two types of models can be considered as macrosystem equilibrium for logit (best-response) dynamics in corresponding congestion games \cite{sandholm2010population}. 

In this paper we popularise the results of \cite{gasnikov2016evolutionary} for english-reading people\footnote{The paper \cite{gasnikov2016evolutionary} was written in Russian and have not been translated yet.} and refine the results on the convergence rate. Moreover, we propose superposition of the considered dynamics to describe equilibrium in two-stage traffic assignment model \cite{de2011modelling,gasnikov2014three}.

One of the main results of the paper is Theorem~\ref{Th:main}, where it is proved that the natural logit-choice and best-response markovian population dynamics in traffic assignment model (congested population game) converge to equilibrium. By using Cheeger's inequality we first time show that mixing time (the time required to reach equilibrium) of these dynamics $T_{mix}$ is proportional to $\ln N$, where $N$ is a total number of agents. Note, that in related works analogues of this theorem were proved without estimating of $T_{mix}$ \cite{sandholm2010population,gasnikov2013vvedenie,gasnikov2014three}. We confirm Theorem~\ref{Th:main} by numerical experiments.

Another important result is a saddle-point reformulation of two-stages traffic assignment model. We explain how to apply results of Theorem~\ref{Th:main} to this model.

\section{Traffic assignment. Problem statement}
\label{sec:beckmann}
Following \cite{kubentayeva2021finding} we describe the problem statement (the next two standard subsections are mainly taken from \cite{kubentayeva2021finding}, starting from the description of <<Population games dynamics ...>> the narration is original). 

Let the urban road network be represented by a directed graph $G = ( V, E )$, where vertices $V$ correspond to intersections or centroids \cite{sheffi1985urban} and edges $E$ correspond to roads, respectively.
Suppose we are given the travel demands: namely, let $d_w$ (veh/h) be a trip rate for an origin-destination pair $w$ from the set $OD \subseteq \{ w = (i, j) : i \in O, \; j \in D \}$. Here $O \subseteq V$ is the set of all possible origins of trips, and $D \subseteq V$ is the set of destination nodes.
For OD pair $w = (i, j)$ denote by $P_w$ the set of all simple paths from $i$ to $j$. Respectively, $P = \bigcup_{w \in OD} P_w$ is the set of all possible routes for all OD pairs. 
Agents travelling from node $i$ to node $j$ are distributed among paths from $P_w$, i.e.\ for any $p \in P_w$ there is a flow $x_p \in \R_+$ along the path $p$, and $\sum_{p \in P_w} x_p = d_w$.
Flows from vertices from the set $O$ to vertices from the set $D$ create the traffic in the entire network $G$, which can be represented by an element of
\[
X = X(d) = \Bigl\{x \in \R_{+}^{|P|} : \; \sum_{p \in P_w} x_p = d_w, \; w \in OD \Bigr\}.
\]
Note that the dimension of $X$ can be extremely large: e.g.\ for $n \times n$ Manhattan network $\log |P| = \Omega(n)$.
To describe a state of the network we do not need to know an entire vector $x$, but only flows on arcs:
\[
f_e(x) = \sum_{p \in P} \delta_{e p} x_p \quad \text{for} \quad e \in E,
\]
where $\delta_{e p} = \mathbbm{1}\{e \in p\}$. Let us introduce a matrix $\Theta$ such that $\Theta_{e, p} = \delta_{e p}$ for $e \in E$, $p \in P$, so in vector notation we have $f = \Theta x$. To describe an equilibrium we use both path- and link-based notations $(x, t)$ or $(f, t)$. 

\textbf{Beckmann model.} One of the key ideas behind the Beckmann model is that the cost (e.g.\ travel time, gas expenses, etc.) of passing a link $e$ is the same for all agents and depends solely on the flow $f_e$ along it. In what follows, we denote this cost for a given flow $f_e$ by $t_e = \tau_e(f_e)$.
Another essential point is a behavioral assumption on agents called the first Wardrop's principle: we suppose that each of them knows the state of the whole network and chooses a path $p$ minimizing the total cost
\[
T_p(t) = \sum_{e \in p} t_e.
\]

The cost functions are supposed to be continuous, non-decreasing, and non-negative. Then $(x^*, t^*)$, where $t^* = (t_e^*)_{e \in E}$, is an equilibrium state, i.e.\ it satisfies conditions
\begin{gather*}
    t_e^* = \tau_e(f_e^*), \quad\text{where}\quad f^* = \Theta x^*, \\ 
    x^*_{p_w} > 0 \Longrightarrow T_{p_w}(t^*) = T_w(t^*) = \min_{p \in P_w} T_p(t^*),
\end{gather*}
if and only if $x^*$ is a minimum of the potential function:
\begin{align*}
    \Psi\left(f(x)\right) = \sum_{e \in E} \underbrace{\int_{0}^{f_e} \tau_e (z) d z}_{\sigma_e(f_e)} 
    \longrightarrow \min_{f = \Theta x, \; x \in X} \\
    \Longleftrightarrow \Psi(f) = \sum_{e \in E} \sigma_e (f_e) 
    \longrightarrow \min_{f = \Theta x : \; x \in X}, \tag{B}\label{PrimalBeckmann}
\end{align*}
and $t_e^* = \tau_e(f_e^*)$ \cite{beckmann1956studies}.

Another way to find an equilibrium numerically is by solving a dual problem.
According to Theorem~4 from \cite{nesterov2003stationary,gasnikov2014three}, we can construct it in the following way:
   $$ \min_{f=\Theta x : \; x \in X} \Psi(f)  $$  $$ = \min_{x\in X, \; f} \left [ \Psi(f) + \sup_{t \in \R^{|E|}} \langle t, \Theta x - f \rangle \right ] = \sup_{t \in \R^{|E|}} \min_{x \in X, \; f} \left [ \Psi(f) + \langle t, \Theta x - f \rangle \right ] $$  
     $$ = \sup_{t \in \R^{|E|}} \left [ - \sum_{e \in E} \max_{f_e} \{t_e f_e - \sigma_e(f_e) \} + \min_{x \in X} \sum_{p\in P} \sum_{e \in E} t_e \delta_{ep} x_p \right ]$$  
     $$ = \max_{t \in \text{dom} \sigma^*} - \left [ \sum_{e \in E} \sigma_e^*(t_e) - \sum_{w \in OD} d_w T_w(t) \right ],  $$ 
where 
\[
\sigma_e^*(t_e) = \sup_{f_e \ge 0} \{t_e f_e - \sigma_e(f_e) \} 
= \fbar_e \left( \frac{t_e - \tbar_e}{\tbar_e \rho} \right)^{\mu} \frac{\left(t_e - \bar{t}_e \right)}{1 + \mu}
\] 
is the conjugate function of $\sigma_e(f_e)$, $e \in E$.
Finally, we obtain the dual problem, which solution is $t^*$:
\begin{equation}\label{DualBeckmann}
   \max_{t \ge \bar{t}}\left\{\sum_{w \in OD} d_w T_w(t) - \sum_{e \in E} \sigma_e^*(t_e)\right\}. \tag{DualB}
\end{equation}

When we search for the solution to this problem  numerically, on every step of an applied method we can reconstruct primal variable $f$ from the current dual variable $t$: $f \in \partial \sum_{w \in OD} d_w T_w(t)$. 
This condition reflects the fact that every driver choose the shortest route. Another condition $t_e = \tau_e(f_e)$ can be equivalently rewrite as $f_e =\frac{d}{dt_e} \sigma_e^*(t_e)$.
This condition with the condition $f \in \partial \sum_{w \in OD} d_w T_w(t)$ form the optimization problem \eqref{DualBeckmann}.

If $\mu\to 0+$ Beckmann's model will turn into Nesterov--dePalma model \cite{gasnikov2014three,kotlyarova2022proof}.

\textbf{Population games dynamic for (stochastic) Beckmann model.} 
 Let us consider each driver to be an agent in population game, where $P_w$, $w \in OD$ is a set of types of agents. All agent (drivers) of type $P_w$ can choose one of the strategy $p\in P_w$ with cost function $T_p\left(t\left(f\left(x\right)\right)\right):=\tilde{T}_p(x)$. Assume that every driver / agent independently of anything (in particular of any other drivers)  is considering the opportunity to reconsider his choice of route / strategy $p$ in time interval $[\text{t},~\text{t}+\Delta \text{t})$ with probability $\lambda \Delta \text{t} + o(\Delta \text{t})$, where $\lambda > 0$ is the same for all drivers / agents. It means that with each driver we relate its own Poisson process with parameter $\lambda$.
If in moment of time $\text{t}$ (when the flow distribution vector is $x(\text{t})$) the the driver of type $P_w$ decides to reconsider his route, than he choose the route $q \in P_w$ with probability
\begin{equation}\label{eq:gumbel}
p_q\left(\tilde{T}\left(x\left(\text{t}\right)\right)\right) = \mathds{P}\left(q = \argmax_{p\in P_w; j=1,...,J}\left\{-\tilde{T}_p\left(x\left(\text{t}\right)\right) + \xi_{p,j}\right\}\right),
\end{equation}
where $\xi_{p,j}$ are i.i.d. and satisfy Gumbel $\max$ convergence theorem \cite{leadbetter1983asymptotic} when $J\to \infty$ with the parameter $\gamma$ (e.g. $\xi_{p,j}$ has (sub)exponential tails at $\infty$). It means that  $\xi_p = \max_{j=1,..,J} \xi_{p,j}$ asymptotically (when $J\to\infty$) has Gumbel distribution $\mathds{P}\left(\xi_p < \xi\right) = \exp\left(-\exp\left(-\xi/\gamma - E\right)\right)$, where $E\simeq 0.5772$ is Euler constant. Note that $\mathds{E} \xi_p = 0$, $\text{Var}~\xi_p  = \pi^2\gamma^2/6$. In words \eqref{eq:gumbel} means that every driver try to choose the best route. But the only available information are noise corrupted values $\tilde{T}_p$. So the driver try to choose the best route focused on the worst forecasts for each route. 

One of the main results of Discrete Choice Theory is as follows \cite{anderson1992discrete}  
\begin{equation}\label{eq:logit}
p_q\left(\tilde{T}\left(x\left(\text{t}\right)\right)\right) = \frac{\exp\left(-\tilde{T}_p\left(x\left(\text{t}\right)\right)/\gamma\right)}{\sum_{q\in P_w} \exp\left(-\tilde{T}_q\left(x\left(\text{t}\right)\right)/\gamma\right)},
\end{equation}
where $p_q\left(\tilde{T}\left(x\left(\text{t}\right)\right)\right)$ was previously defined in \eqref{eq:gumbel}.

Note that the described above dynamic degenerates into the best-response dynamic when $\gamma \to 0 +$ \cite{sandholm2010population}.

\begin{theorem}\label{Th:main} Let $\sum_{p \in P} x_p = N$. For all $x(0) \in X$ there exists such a constant $c\left(x(0)\right)$ that for all $\sigma \in \left(0,~0.5\right)$ and $\text{t} \ge T_{mix}= c\left(x(0)\right)\lambda^{-1} \ln N$:
\begin{equation}\label{eq:hoeffding}
    \mathds{P}\left(\left\|\frac{x(\text{t})}{N} - x^{*} \right\|_2\le \frac{2\sqrt{2}+ 4\sqrt{\ln\left(\sigma^{-1}\right)}}{\sqrt{N}}\right) \ge 1 - \sigma,
\end{equation}
    where 
\begin{equation}\label{eq:equil}
    x^* = \argmin_{x \in X\left(d/N\right)}\left\{\tilde{\Psi}\left(f(x)\right) + \gamma\sum_{p\in P} x_p\ln x_p\right\},
\end{equation}    
    \begin{center}
        $\tilde{\Psi}\left(f(x)\right) = \sum_{e \in E} \int_{0}^{f_e} \tilde{\tau}_e (z) d z,$\quad $\tilde{\tau}_e(z) = \tau_e(zN).$
    \end{center}
\end{theorem}
\begin{proof}
    The first important observation is that the described Markov process is reversible. That is it satisfies Kolmogorov's detailed balance condition (see also \cite{malyshev2008reversibility}) with stationary (invariant) measure
    $$\pi(x) = \frac{N!}{x_1!...x_{|P|}!}\exp\left(-\frac{\Psi\left(f(x)\right)}{\gamma}\right),$$
    where $x \in X(d)$ \cite{sandholm2010population}. The result of type \eqref{eq:hoeffding} for $x(\infty)$ holds true due to Hoeffding's inequality in a Hilbert space \cite{boucheron2013concentration}. We can apply this inequality for multinomial part $\frac{N!}{x_1!...x_{|P|}!}$. The rest part may only strength the concentration phenomenon, especially when $\gamma$ is small. The Sanov's theorem \cite{thomas2006elements} says that $x^*$ from \eqref{eq:equil} asymptotically ($N \to \infty$) describe the proportions in maximum probability state, that is 
    $$x^*N\simeq\argmax_{x\in X(d)}\frac{N!}{x_1!...x_{|P|}!}\exp\left(-\frac{\Psi\left(f(x)\right)}{\gamma}\right).$$
    
    To estimate the mixing time $\sim \lambda^{-1} \ln N$ of the considered Markov process we will put it in accordance with this continuous-time process discrete-time process with step $\left(\lambda N\right)^{-1}$, which corresponds to the expectation time between two nearest events in continuous-time dynamic. Also we consider this discrete Markov chain as a random walk on a proper graph $G = \langle V_G,~E_{G}\rangle$ with starting point corresponds to the vertex $s$ and transition probability matrix $P = \| p_{ij}\|_{i,j=1}^{|V_G|}$. According to a Cheeger's inequality mixing time $\text{t}_{\text{mix}}$ for such a random walk, which approximate stationary measure $\pi$ with accuracy $\varepsilon = O\left(N^{-1/2}\right)$ (in this case  $x\left(\text{t}_{\text{mix}}\right) \simeq x(\infty)$), is 
    $$O\left(\left(\lambda N\right)^{-1}h(G)^{-2}\left(\ln\left(\pi(s)\right) + \ln(\varepsilon^{-1})\right)\right),$$
    where Cheeger's constant is determined as
    $$h(G) = \min_{S \subseteq V_G: \pi(S)\le 1/2}  \mathds{P}\left(S\to\bar{S}\mid S\right) = \min_{S \subseteq V_G: \pi(S)\le 1/2} \frac{\sum_{(i,j)\in E_G,~i\in S,~j\in\bar{S}}\pi(i)p_{ij}}{\sum_{i\in S} \pi(i)},$$
    where $\bar{S} = V_G/S$ \cite{levin2017markov}.
    Since $G$ and $P$ correspond to reversible Markov chain with stationary measure $\pi$ that exponentially concentrate around $x^*N$ one can prove that isoperimetric  problem of finding optimal set of vertexes $S$ has the following solution, which we described below roughly -- up to a numerical constant: $S$ is a set of such states $x \in X(d)$ that $\|x - x^*N\|_2 \lesssim O\left(\sqrt{N}\right)$. Since the the ratio of sphere volume of radius $O(\sqrt{N})$ to the volume of the ball of the same radius is $O(N^{-1/2})$, we can obtain that $h(G)\sim N^{-1/2}$. So up to a $\ln\left(\pi(s)\right)$ (we put it into $c\left(x(0)\right)$) mixing time is indeed $\sim \lambda^{-1} \ln N$. 
\end{proof}

Note that the describe above approach assumes that we first $\text{t}\to \infty$ and then $N\to \infty$. If we firstly take $N\to \infty$ than due to Kurtz's theorem \cite{ethier2009markov} $c(\text{t})=\lim_{N\to \infty} x(\text{t})/N$ satisfies (for all $w \in OD$, $p \in P_w$)
$$\frac{dc_p}{dt} = \bar{d}_w\frac{\exp\left(-\bar{T}_p\left(c\left(\text{t}\right)\right)\right)}{\sum_{q\in P_w} \exp\left(-\bar{T}_q\left(c\left(\text{t}\right)\right)\right)} - c_p(\text{t}),$$
where $\bar{d} =\lim_{N\to \infty} d/N$, $\bar{T}_p\left(c(\text{t})\right) = \tilde{T}_p\left(x(\text{t})\right)$. Note that Sanov's type function $\tilde{\Psi}\left(f(c)\right) + \gamma\sum_{p\in P} c_p\ln c_p$ from \eqref{eq:equil} will be Boltzmann--Lyapunov type function for this system of ordinary differential equations (SODE), that is decrees along the trajectory of SODE. This result is a particular case of general phenomenon: Sanov's type function for invariant measure obtained from Markovian dynamics is Boltzmann--Lyapunov type function for deterministic Kurtz's kinetic dynamics \cite{batishcheva20052,malyshev2008reversibility,gasnikov2013entropy}.

\section{Origin-destination matrix estimation}
\label{sec:OD}
Origin-destination matrix estimation model can be considered as a particular case of the traffic assignment model. The following interpretation goes back to \cite{gasnikov2014three,gasnikov2016evolutionary}. Indeed, let us consider fictive origin $o$ and fictive destination $d$. So $\tilde{O} = \{o\}$, $\tilde{D} = \{d\}$. Let us draw fictive edges from $o$ to real origins of trips $O$. The cost of the trip at edge $(o,i)$ is $\lambda^O_i$ -- an average price that each agent pays to live at this origin region $i\in O$. Analogously, let us draw edges from the vertexes of the real destination set $D$ to $d$. The cost of the trip at edge $(j,d)$ is $-\lambda^D_j$ -- minus average salary that each agent obtain in destination region $j\in D$. So the set of all possible routes (trips) from $o$ to $d$ can be described by pairs $(i,j)\in OD$. Each route consist of three edges $o \to i$ with cost $\lambda^O_i$, edge $i \to j$ with cost $T_{ij}$ (is available as an input of the model) and edge $j \to d$ with cost $-\lambda^D_j$. So equilibrium origin-destination matrix $d = \{d_{ij}\}_{(i,j)\in OD}$ (up to a scaling factor) can be find from entropy-linear programming problem
\begin{equation}\label{eq:gasnikova}
\min_{d\ge 0:~\sum_{(i,j)\in OD} d_{ij}=1}  \sum_{i\in O} \lambda^O_i \sum_{j \in D} d_{ij} -   \sum_{j\in D} \lambda^D_j \sum_{i \in O} d_{ij} + \sum_{(i,j)\in OD} T_{ij}d_{ij} + \gamma\sum_{(i,j)\in OD} d_{ij}\ln d_{ij}.
  \end{equation}
In real life $\lambda^O_i$ and $\lambda^D_j$ are typically unknown. But at the same time the following agglomeration characters are available
\begin{equation}\label{eq:O}
    \sum_{j\in D} d_{ij} = L_i,\quad i \in O,
\end{equation}
\begin{equation}\label{eq:D}
    \sum_{i\in O} d_{ij} = W_j,\quad j \in D.
\end{equation}
The key observation is that \eqref{eq:gasnikova} can be considered as Lagrange multipliers principle for constraint entropy-linear programming problem 
\begin{equation}\label{eq:wilson}
\min_{d\ge 0:~d\text{ satisfies } \eqref{eq:O},\eqref{eq:D}}   \sum_{w\in OD} T_w d_w + \gamma\sum_{w\in OD} d_w\ln d_w,
  \end{equation}
where $\lambda^O_{i}$ and $\lambda^D_j$ is Lagrange multipliers for \eqref{eq:O} and \eqref{eq:D} correspondingly. The last model is called Wilson's entropy origin-destination matrix model \cite{wilson2013entropy,gasnikov2013vvedenie}.

The result of Theorem~\ref{Th:main} can be applied to this model due to the mentioned above reduction.

\section{Two-stages traffic assignment model}
\label{sec:two}
From the Section~\ref{sec:beckmann} we may know that Beckmann's model requires origin-destination matrix as an input $\{d_{w}\}_{w\in OD}$. So Beckmann's model allows to calculate $t(d)$. At the same time from the Section~\ref{sec:OD} we may know that Wilson's entropy origin-destination model requires cost matrix $\{T_w\}_{w\in OD}$ as an input, where $T_w:=T_w(t) = \min_{p\in P_w} T_p(t)$. So Wilson's model allows to calculate $d\left(T(t)\right)$. The solution of the system $d = d\left(T\left(t(d)\right)\right)$ is called two-stage traffic assignment model \cite{de2011modelling}. Following \cite{gasnikov2014three,gasnikov2020traffic} we can reduce this problem to the following one (see \eqref{DualBeckmann} and \eqref{eq:wilson})
  \begin{equation}\label{eq:TS}
       \min_{d\ge 0:~d\text{ satisfies } \eqref{eq:O},\eqref{eq:D}}\left\{  \max_{t \ge \bar{t}}\left\{\sum_{w \in OD} d_w T_w(t) - \sum_{e \in E} \sigma_e^*(t_e)\right\} + \gamma\sum_{w\in OD} d_w\ln d_w \right\}.
  \end{equation}
  The problem \eqref{eq:TS} can be rewritten as a convex-concave (if $\tau_{e}'(t_e) \ge 0$) saddle-point problem (SPP)
    \begin{equation}\label{eq:SPP}
       \min_{d\ge 0:~d\text{ satisfies } \eqref{eq:O},\eqref{eq:D}}\max_{t \ge \bar{t}}\left\{\sum_{w \in OD} d_w T_w(t) - \sum_{e \in E} \sigma_e^*(t_e) + \gamma\sum_{w\in OD} d_w\ln d_w\right\}.
  \end{equation}
 This SPP can be efficiently solved numerically \cite{gasnikov2015universal,gasnikov2020traffic}.
 
 Note that if we consider best-response dynamics from Section~\ref{sec:beckmann} with the parameter $\lambda: = \lambda_{\text{Beck}}$ and logit dynamic with the parameter  $\lambda: = \lambda_{\text{Wil}}$ for origin-destination matrix estimation and assume that $\lambda_{\text{Beck}}\gg \lambda_{\text{Wil}}$ than such a dynamic will converge to the stationary (invariant) measure that is concentrated around the solution of SPP problem \eqref{eq:SPP}. This result can be derived from the more general result related with hierarchical congested population games \cite{dvurechensky2016primal}.
 
\section{Numerical experiments}
The main result of the paper is Theorem~\ref{Th:main}. The main new result of this theorem is a statement that mixing time of the considered markovian logit-choice and best-response dynamics $T_{mix}$ is approximately $c_1 + c_2\cdot\ln N$, where $N$ is a number of agents.

\begin{figure}[H]
    \centering
    \centering
    \includegraphics[scale=.5]{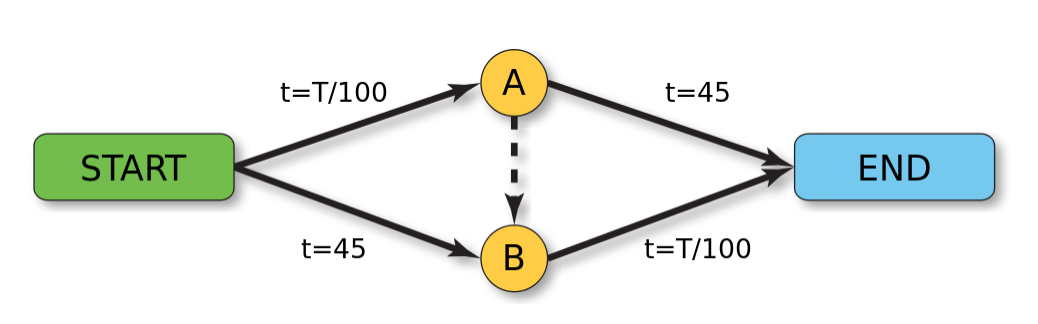}
    \caption{Braess's paradox graph}\label{fig:BP}
\end{figure}

We consider Braess's paradox example \cite{frank1981braess}, see figure~\ref{fig:BP}. This picture is taken from Wikipedia.  Here Origin is START and Destination is END. We have one OD-pair and put $d = N$ -- the number of agents. The <<paradox>>  arises when $N = 4000$. In this case when there is no road from A to B  we have two routes (START, A, END) and (START, B, END) with 2000 agents at each route. So the equilibrium time costs at each route will be 65. When the road AB is present (this road has time costs 0) all agents will use the route (START, A, B, END) and this equilibrium has time costs 80. That is paradoxically larger than it was without road AB. 

In series of experiments (see figures~\ref{fig:0.1},~\ref{fig:0.01},~\ref{fig:BR}) the dependence of mixing time $T_{mix}$ from $\ln N$ was investigated. Details see in

\url{https://github.com/ZVlaDreamer/transport_flows_project}.

 \begin{figure}[H]
    \centering
    \centering
    \includegraphics[scale=.5]{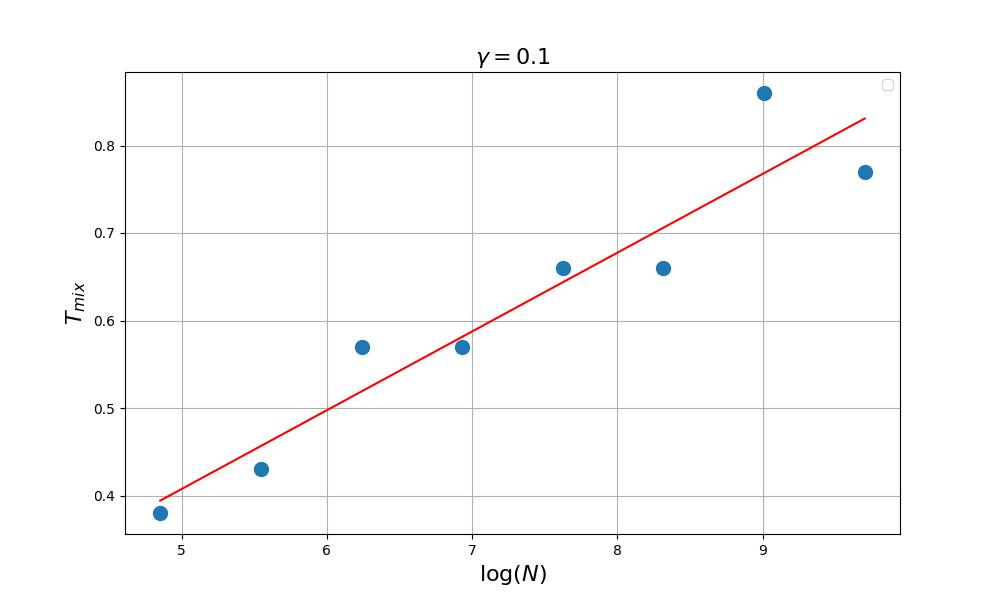}
    \caption{Logit-choice dynamic $\gamma = 0.1$}\label{fig:0.1}
\end{figure}

 \begin{figure}[H]
    \centering
    \centering
    \includegraphics[scale=.5]{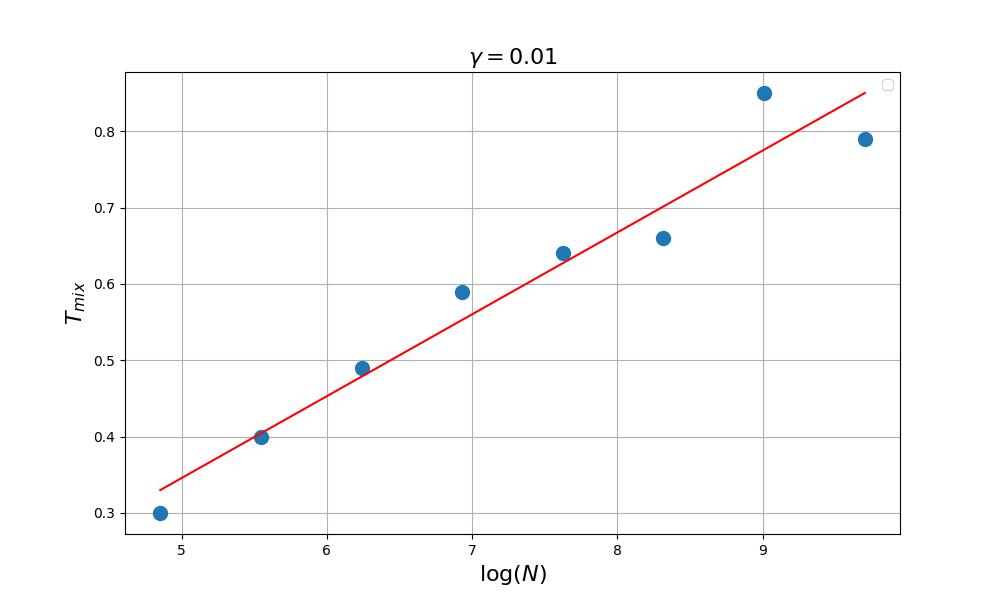}
    \caption{Logit-choice dynamic $\gamma = 0.01$}\label{fig:0.01}
\end{figure}

 \begin{figure}[H]
    \centering
    \centering
    \includegraphics[scale=.5]{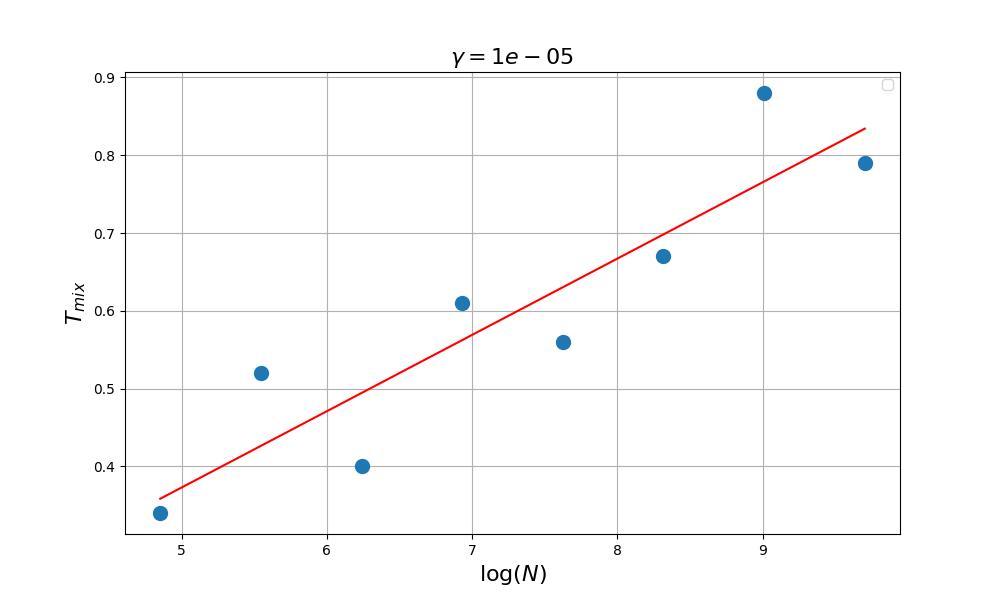}
    \caption{Best response dynamic (as a limit of logit-choice dynamics $\gamma \to 0+$)}\label{fig:BR}
\end{figure}

Numerical experiments confirm Theorem~\ref{Th:main}. Note that in \cite{gasnikov2013vvedenie} it was described a real-life experiment oraganized with MIPT students in Experimental Economics Lab. The students were agents and play in repeated Braess's paradox game. The result of experiments from \cite{gasnikov2013vvedenie} is also well agreed with the described above numerical experiments.

\section{Conclusion}
In this paper we investigate logit-choice and best-response population markovian dynamics converges to equilibrium in corresponding traffic assignment model. We show that mixing time is proportional to logarithm from the number of agent. Numerical experiments confirm  that the dependence is probably unimprovable. We also consider two-stage traffic assignment model and describe how to interpret equilibrium for this model in an evolutionary manner.

We dedicate this paper to our colleague prof. Vadim Alexandrovich Malyshev (April 13, 1938 — September 30, 2022). We express our gratitude to Leonid Erlygin (MIPT) and Vladimir Zholobov (MIPT) who conducted numerical experiments.

The work of E. Gasnikova was supported by the Ministry of Science and Higher Education of the Russian Federation (Goszadaniye) 075-00337-20-03, project No. 0714-2020-0005. The work of A. Gasnikov was supported by the strategic academic leadership program <<Priority 2030>> (Agreement  075-02-2021-1316 30.09.2021).

\bibliographystyle{unsrt}
\bibliography{Lib_new}
\end{document}